\documentclass[12pt]{amsart}
\usepackage{amssymb,amsmath,amscd,enumerate}
\usepackage{tikz}

\usepackage{capt-of}
\usepackage[colorlinks=true,linkcolor=blue,citecolor=blue]{hyperref}
\usepackage{graphicx}
\usepackage{hyperref}

\usepackage{float} 
\usepackage{mathrsfs}
\usetikzlibrary{arrows}
 \numberwithin{equation}{section}
 \newtheorem{thm}{\bf Theorem}[section]
 \newtheorem{lem}[thm]{\bf Lemma}
 \newtheorem{conj}[thm]{\bf Conjecture}
 \newtheorem{cor}[thm]{\bf Corollary}
 
 \theoremstyle{definition}
 \newtheorem{defn}[thm]{Definition}
 
 \newtheorem{exam}[thm]{Example}
 
 \newtheorem{rem}[thm]{Remark}
 
\newtheorem*{thm*}{Theorem}
\usepackage{color}

\DeclareMathOperator{\supp}{supp}
\DeclareMathOperator{\reg}{reg}

\DeclareMathOperator{\diam}{diam}
\DeclareMathOperator{\dist}{dist}

\DeclareMathOperator{\im}{im}

\def\m{\mathfrak m}

\def\M{\mathcal M}

\newcommand {\kk} {\Bbbk} 

\usepackage[normalem]{ulem}

\begin{document}
\title[Edge ideals of powers of graphs]{Regularity of edge ideals of powers of graphs}
\date{}

\author[My Hanh Pham]{My Hanh Pham}
\address{Faculty of Education, An Giang University\\Vietnam National University Ho Chi Minh City\\Dong Xuyen, Long Xuyen, An Giang, Vietnam}
\email{pmhanh@agu.edu.vn}

\author[Thanh Vu]{Thanh Vu}
\address{Institute of Mathematics, VAST, 18 Hoang Quoc Viet, Hanoi, Vietnam}
\email{vuqthanh@gmail.com}

\let\thefootnote\relax\footnotetext{*Corresponding Author: Thanh Vu }

 
\begin{abstract} 
We prove that the regularity of edge ideals of powers of forests is weakly decreasing. We then compute the regularity of edge ideals of powers of cycles.
 \end{abstract}

\maketitle

\section{Introduction}
  \label{sec1} 
Let $G$ be a simple graph on the vertex set $V(G) = [n] = \{1,\ldots,n\}$ and the edge set $E(G)$. The distance $\dist_G(u,v)$ between the vertices $u$ and $v$ of $G$ is the number of edges in a shortest path connecting $u$ and $v$. If no such path exists, $\dist_G(u,v)=\infty$. The diameter of $G$, denoted by $\diam(G)$, is the greatest distance between any two vertices of $G$. For a positive integer $d$, the $d$-th power of $G$, denoted by $G^d$, is a graph on the vertex set $V(G)$ and $\{u,v\}$ is an edge of $G^d$ if and only if $\dist_G(u,v) \le d$. In particular, $G^d$ is a complete graph for any $d \ge \diam(G)$. Graph power is a classical concept in graph theory \cite{BM}. It deeply connects to information theory, communication complexity, and Ramsey theory \cite{AL}. Recently, it also appears in the work of Cooper, El Khoury, Faridi, Mayes-Tang, Morey, Sega, and Spiroff \cite{CEFMMSS} where the authors studied free resolutions of powers of squarefree monomial ideals. In this work, we study the regularity of edge ideals of powers of graphs. Recall that the edge ideal of $G$ over a field $\kk$ is defined by 
$$I(G) = (x_ix_j \mid \{i,j\} \in E(G)) \subset R = \kk[x_1,\ldots,x_n].$$
We denote by $\im(G)$ the induced matching number of $G$. In Lemma \ref{ind_math_powers}, we prove that $\im(G^{d+1}) \le \im(G^d)$ for any positive integer $d$. Although the induced matching number of $G$ only gives a lower bound for the regularity of its edge ideal \cite[Lemma 2.2]{K}, based on computations, we suggest the following.
\begin{conj}\label{conj_reg} Let $G$ be a simple graph and $d$ be a positive integer. Then 
$$\reg \left ( R/I \left ( G^d \right ) \right ) \ge \reg \left ( R/I \left ( G^{d+1} \right ) \right ),$$
where $\reg$ denotes the Castelnuovo-Mumford regularity.
\end{conj}
In the first main result of the paper, we prove that Conjecture \ref{conj_reg} holds for forests.

\begin{thm}\label{thm_reg_decrease_forests} Let $G$ be a forest and $d$ be a positive integer. Then 
$$\reg \left ( R/I \left ( G^d \right ) \right ) \ge \reg \left ( R/I \left ( G^{d+1} \right ) \right ).$$    
\end{thm}
If $G = G_1 \cup G_2 \cup \cdots \cup G_t$ where $G_i$ are the connected components of $G$ then $G^d = G_1^d \cup G_2^d \cup \cdots \cup G_t^d$. By the results of Nguyen and Vu \cite{NV}, we may assume that $G$ is connected when studying the homological invariants of powers of $I(G^d)$. When $d \ge 2$, $G^d$ contains many triangles and is generally more complicated than $G$. Fortunately, when $G$ is a tree, by the result of Laskar and Shier \cite{LS}, $G^d$ is chordal for all $d \ge 1$. The conclusion then follows from Lemma \ref{ind_math_powers} and \cite[Theorem 14]{W}.

On the other hand, the power operation preserves symmetries of $G$, e.g. if $G$ is circulant then its powers are circulant. Our next result shows that Conjecture \ref{conj_reg} also holds for cycles. 
 
\begin{thm}\label{thm_reg_main} Let $n,d$ be positive integers such that $n \ge 3$. Then
$$\reg \left ( R / I(C_n^d) \right ) = \begin{cases} 1 & \text{ if } n \le 2d+2,\\
    2 & \text { if } n = 2d + 3, \\
    \left \lfloor \frac{n}{d+2} \right \rfloor & \text { otherwise},
\end{cases}$$
where $C_n$ denotes a cycle on $n$ vertices.
\end{thm}
This result confirms the computational experiments done by Romeo \cite{Ro}. In the cases where $n \le 2d+3$ the results are essentially known. In the cases where $n = k(d+2) + r$ for $0 \le r \le d$, we drop the vertices $1, 2, \ldots, d$ sequentially to get to the case of powers of paths. The critical case where $n = k(d+2) + d+1$ is much more subtle. The basic idea is simple, though. To establish the inequality $\reg (J) \le k+1$ for a monomial ideal $J$, we use one of the following strategies. 
\begin{enumerate}
    \item  Express $J = K + x L$ for some variable $x$, then establish that $\reg (K), \reg (K\cap L) \le k+1$ and $\reg (L) \le k$.
    \item Choose a sequence of variables $y_1,\ldots, y_{\ell}$ so that $J + (y_1, \ldots, y_\ell)$ is well understood and have regularity at most $k+1$ and $(J + (y_1, \ldots, y_{j-1}) ) : y_j$ is well understood and have regularity at most $k$ for all $j = 1, \ldots, \ell$.
\end{enumerate}
We refer to Section \ref{sec3} for more details. 

In Section \ref{sec2}, we recall the notion of graph powers and prove that the induced matching numbers of powers of graphs are weakly decreasing. We then deduce Theorem \ref{thm_reg_decrease_forests}. In Section \ref{sec3}, we establish Theorem \ref{thm_reg_main}.

\section{Powers of graphs}\label{sec2}

 \subsection{Graph} 

 \begin{defn} Let $G$ be a simple graph with the vertex set $V(G) = \{x_1,\ldots,x_n\}$ and the edge set $E(G)$. 

\begin{enumerate}
    \item The neighborhood of a vertex $x \in V(G)$ is $N_G(x)=\{y\in V(G) \mid \{x,y\}\in E(G)\}$. The closed neighborhood of $x$ is $N_G[x]=N_G(x)\cup\{x\}$. 
    \item For a subset $U \subset V(G)$, the neighborhood and closed neighborhood of $U$ in $G$ are defined respectively by $N_G(U) = \cup \left ( N_G(x) \mid x \in U \right )$ and $N_G[U] = \cup \left ( N_G[x] \mid x \in U \right )$. 
    \item A simple graph $H$ is a subgraph of $G$ if $V(H) \subseteq V(G)$ and $E(H) \subseteq E(G)$. $H$ is an induced subgraph of $G$ if it is a subgraph of $G$ and $E(H) = E(G) \cap V(H) \times V(H)$.
    \item For a subset $U \subset V(G)$, $G[U]$ and $G\backslash U$ denote the induced subgraph of $G$ on $U$ and on $V(G) \setminus U$ respectively.
    \item A path $P_n$ on $n$ vertices is the graph on $V(P_n) = \{x_1, \ldots, x_n\}$ with the edge set $E(P_n) = \{ \{x_1,x_2\},\ldots, \{x_{n-1},x_n\} \}.$
    \item A cycle $C_n$ on $n$ vertices is the graph on $V(C_n) = \{x_1,\ldots,x_n\}$ with the edge set $E(C_n) = E(P_n) \cup \{ \{x_1,x_n\} \}$.
    \item A forest is a graph without any cycles. A tree is a connected forest.
    \item $G$ is called chordal if it has no induced cycles of length at least four.
    \item A subset $M \subseteq E(G)$ is called a matching of $G$ if no two edges in $M$ share a common vertex. It is an \textit{induced matching} if $M$ forms an induced subgraph of $G$. \textit{Induced matching number} of $G$, denoted by $\im(G)$, is the largest size of an induced matching of $G$.
\end{enumerate}
\end{defn}

The induced matching number of the $d$-th power of an $n$-cycle was given in \cite[Corollary 3.4]{Ro}. We now compute the induced matching number of the $d$-th power of a path on $n$ vertices.


\begin{lem}\label{lem_im_powers_paths} Let $n,d$ be positive integers such that $n \ge 2$. Then 
$$\im(P_n^d) = \left \lfloor \frac{n+d}{d+2} \right \rfloor.$$    
\end{lem}
\begin{proof} Fix $d \ge 1$. When $n \le d+1$, $P_n^d$ is a complete graph on $n$ vertices. Hence, $\im(P_n^d) = 1$. Now, assume that $n \ge d+2$. Write $n = k(d+2) + r$ for some $k \ge 1$ and $0 \le r \le d+1$. Let 
$$\mathcal{M}= \{ \{1, 2\}, \{ (d+2)+1, (d+2)+2\}, \ldots, \{ (d+2) \ell +1, (d+2) \ell + 2\}\}, $$
where $\ell = k-1$ if $r \le 1$ and $\ell = k$ if $r \ge 2$. Then $\M$ is an induced matching of $P_n^d$. Hence, $\im(P_n^d)\ge |\mathcal{M}|=\lfloor\frac{n+d}{d+2}\rfloor$. We now prove by induction on $n$ that $\im(P_n^d) \le \left \lfloor \frac{n+d}{d+2} \right \rfloor$. The base case where $n \le d+1$ is clear. Now, assume that $n \ge d+2$. Let $\mathcal{M}$ be a maximum induced matching of $P_n^d$. If $1 \notin \mathcal{M}$, then $\M$ is an induced matching of $P_{n-1}^d$, and the conclusion follows from induction. Thus, we may assume that $1 \in \M$. Assume that $\{1,i\} \in \M$ for some $i \le d+1$. Then dropping $\{1,i\}$ from $\M$ gives an induced matching of $P_{n-i-d}^d$. By induction, we deduce that 
$$\im(P_n^d) \le 1 + \im (P_{n-i-d}^d) \le 1 + \im (P_{n-d-2}^d) =  \left \lfloor
        \frac{n +d}{d+2} \right \rfloor.$$
        The conclusion follows.
\end{proof}

In order to prove that $\im(G^{d+1})\le \im(G^d)$ for an arbitrary graph $G$, we introduce the following concept.

\begin{defn} Let $u,v$ be vertices of a simple graph $G$. Assume that $\dist_G(u,v) = d \ge 2$. We denote by $W(u,v)$ the set of vertices $w$ of $G$ such that $\dist_G(u,w) = d-1$ and $\dist_G(w,v) = 1$. We call $W(u,v)$ the \textit{distance witness set} of $u$ and $v$. We denote by $W[u,v] = W(u,v) \cup \{u,v\}$.
\end{defn}

We now have some simple lemmas. 

\begin{lem}\label{lem_distance_witness} Let $\{u_1,v_1\}$ and $\{u_2,v_2\}$ be an induced matching of $G^d$ such that $\dist_G(u_1,v_1) = d$. Then 
\begin{enumerate}
    \item if $\dist_G(u_2,v_2) \le d-1$ then $W[u_1,v_1] \cap \{u_2,v_2\} = \emptyset$;
    \item if $\dist_G(u_2,v_2) = d$ then $W[u_1,v_1] \cap W[u_2,v_2] = \emptyset$.     
\end{enumerate}
\end{lem}
\begin{proof} By definition, $\dist_G(u_1,u_2), \dist_G(u_1,v_2), \dist_G(v_1,u_2), \dist_G(v_1,v_2) > d$. Hence, $W[u_1,v_1] \cap \{u_2,v_2\} = \emptyset$ and (1) follows immediately. Now, assume that $\dist_G(u_2,v_2) = d$ and there exists an $w \in W(u_1,v_1) \cap W(u_2,v_2)$. By definition, $\dist_G(u_1,w) = d-1$ and $\dist_G(w,v_2) = 1$. Hence, $\dist_G(u_1,v_2) \le d$, which is a contradiction. The conclusion (2) follows.    
\end{proof}

\begin{lem}\label{ind_math_powers} Let $G$ be a simple graph and $d$ be a positive integer. Then $\im(G^{d+1}) \le \im(G^d)$.    
\end{lem}
\begin{proof}
    Assume that $\im \left ( G^{d+1} \right ) = m$ and $\{\{u_1, v_1\}, \ldots, \{u_m,v_m\}\}$ is an induced matching of $G^{d+1}$. We may assume that $\dist_G(u_i,v_i) \le d$ for $i \le s$ and $\dist_G(u_i,v_i) = d+1$ for $i = s+1, \ldots, m$ for some integer $s$ such that $0 \le s \le m$. For $i = s+1, \ldots, m$, let $w_i$ be any element of $W(u_i,v_i)$. By the proof of Lemma \ref{lem_distance_witness}, $\{\{u_1, v_1\}, \ldots, \{u_s,v_s\}, \{u_{s+1}, w_{s+1}\}, \ldots, \{u_m, w_m\}\}$ is an induced matching of $G^d$. The conclusion follows. 
\end{proof}

 \subsection{Regularity} Let $R = \kk[x_1,\ldots,x_n]$ be a polynomial ring over a field $\kk$. Denote by $\m = (x_1, \ldots, x_n)$ the maximal homogeneous ideal of $R$. For a finitely generated graded $R$-module $M$, $H^i_\m (M)$ denotes the $i$-th local cohomology of $M$ with respect to $\m$. The Castelnuovo-Mumford regularity of $M$ is defined by 
$$\reg_R(M) = \max \left \{ i + j \mid H^i_\m (M)_j \neq 0 \right \}.$$
For a nonzero homogeneous ideal $I$ of $R$, we have $\reg (I) = \reg (R/I) + 1$.

\begin{cor}\label{cor_reg_powers_path} Let $n,d$ be positive integers. Then 
$$\reg (R/I(P_n^d)) = \left \lfloor \frac{n+d}{d+2} \right \rfloor.$$      
\end{cor}
\begin{proof}
By \cite[Corollary 2]{LS}, $P_n^d$ is chordal. By \cite[Theorem 14]{W}, $\reg(R/I(P_n^d))=\im(P_n^d)$. The conclusion follows from Lemma \ref{lem_im_powers_paths}.    
\end{proof}

\begin{proof}[Proof of Theorem \ref{thm_reg_decrease_forests}] Let $G$ be a forest. By \cite[Corollary 2]{LS}, $G^d$ are chordal for all $d \ge 1$. By \cite[Theorem 14]{W}, $\reg(R/I(G^d))=\im(G^d)$ for all $d \ge 1$. The conclusion follows from Lemma \ref{ind_math_powers}.
\end{proof}

\begin{rem} \begin{enumerate}
    \item There are chordal graphs $G$ such that $G^2$ is not chordal; see \cite[Theorem 6]{LS}. 
    \item Applying \cite[Corollary 3]{LS}, we deduce that Conjecture \ref{conj_reg} also holds for block graphs.
\end{enumerate}
\end{rem}

\begin{exam} Let $G$ be a sunflower graph $S_6$ as in the following figure. Then $G$ is chordal, but $G^2$ is not chordal. Also, we have $\im(G) = 3$, hence $\reg (R/I(G)) = 3$. Using Macaulay2 \cite{M2}, we get $\reg (R/I(G^2)) = 2$. 

 
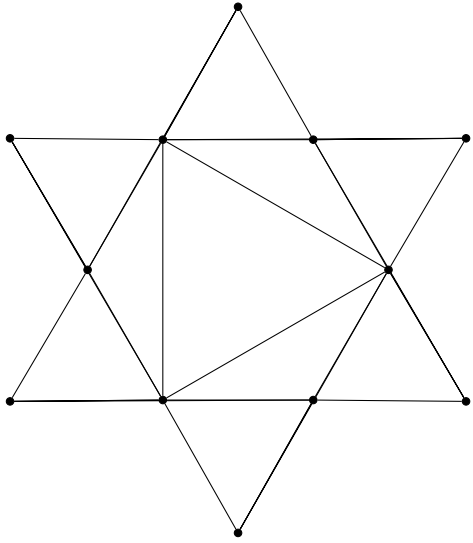
\begin{center}
     
\begin{tikzpicture}
  \foreach \i in {1,...,6}
  {
    \node[draw, fill, circle, inner sep=1pt] (v\i) at (\i*60:2) {};
  }
  
  \foreach \i/\j/\k in {1/2/7, 2/3/8, 3/4/9, 4/5/10, 5/6/11, 6/1/12}
  {
    \node[draw, fill, circle, inner sep=1pt] (v\k) at (60*\i-30:3.5) {};
    \draw (v\i) -- (v\k) -- (v\j);
  }
  
  \foreach \i in {1,...,6}
  {
    \pgfmathtruncatemacro{\j}{mod(\i,6)+1}
    \draw (v\i) -- (v\j);
  }

  \draw (v2) -- (v4);
  \draw (v2) -- (v6);
  \draw (v4) -- (v6);
  \draw (v1) -- (v7);
  \draw (v2) -- (v9);
  \draw (v2) -- (v8);
  \draw (v3) -- (v10);
  \draw (v3) -- (v9);
  \draw (v4) -- (v11);
  \draw (v4) -- (v10);
  \draw (v5) -- (v12);
  \draw (v5) -- (v11);
  \draw (v6) -- (v7);
  \draw (v6) -- (v12);
  \draw (v1) -- (v8);

\end{tikzpicture}

 \end{center}

\end{exam}

\section{Powers of cycles}\label{sec3}
In this section, we compute the regularity of edge ideals of powers of cycles. The results show that Conjecture \ref{conj_reg} also holds for cycles. We first introduce some notations and preliminary results.

If $u\in R$ is a monomial, the \emph{support} of $u$, denoted by $\supp (u)$ is the set of variables $x_i$ such that $x_i | u$. Also, if $J\subset R$ is a monomial ideal with the minimal generating set $G(J)=\{u_1,\ldots,u_m\}$, the \emph{support} of $J$ is $\supp(J)=\bigcup_{i=1}^m \supp(u_i)$. The following results are well known, see e.g. \cite{HH}.

\begin{lem}\label{lem_colon_sum}  Let $I,J,L$ be nonzero monomial ideals and $f$ be a nonzero monomial of $R$. Then 
\begin{enumerate}
    \item $(I+J) : f = (I : f) + (J:f)$.
    \item $(I + J) \cap L = (I \cap L) + (J \cap L)$.
    \item If $\supp (I) \cap \supp (J) = \emptyset$ then $I\cap J = IJ$.
    \item $I \cap (f) = f( I : f).$
\end{enumerate}
\end{lem}

\begin{lem}\label{lem_drop_v} Let $I$ be a nonzero homogeneous ideal and $x$ be a variable of $R$. Then 
$$\reg (I) \le \max \{ \reg (I + (x)), \reg (I : x) + 1\}.$$
In particular, let $G$ be a simple graph and $v$ be a vertex. Then 
$$\reg (I(G)) \le \max \{ \reg (I(G\backslash v)), \reg (I(G \backslash N_G[v])) + 1\}.$$    
\end{lem}

\begin{lem}\label{lem_mul_x} Let $I$ be a nonzero homogeneous ideal and $x$ be a variable of $R$. Then
$$\reg (x I) =   \reg (I) +  1.$$    
\end{lem}

\begin{lem}\label{lem_intersection} Let $I$ and $J$ be nonzero proper homogeneous ideals of $R$. Then 
$$\reg (I + J) \le \max \{ \reg (I), \reg (J), \reg ( I \cap J) - 1\}.$$
\end{lem}

\begin{lem}\label{lem1} Assume that $n = 2d + 3$ and $G = C_n^d$. Then $\reg (R/I(G)) = 2$.
\end{lem}
\begin{proof}
Since $n=2d+3$, we have $C_n^d=C_n(1,2,\ldots,d,\widehat{\lfloor\frac{n}{2}\rfloor})$. Also, $n\neq 2(d+1)$ and ${\rm gcd}(n,d+1)=1$. Hence, by \cite[Theorem 5]{UV}, we obtain $\reg(R/I(G))=2.$
\end{proof}

Assume that $G = C_{n}^d$ where $n \ge 2d+2$. We fix the following notation. 
\begin{align}
    G_0 &= G = C_n^d, \label{eq_no_1}\\
    G_i &= G_{i-1} \setminus \{i\} \text { for } i = 1, \ldots, d. \label{eq_no_2}
\end{align}   
We have 
$$N_{G_{i-1}} (i) = \{n-d + i, \ldots, n, i+1, \ldots, i+d\}.$$
Hence, 
\begin{align}
    G_d & \cong P_{n - d}^d \label{eq_3_1}\\
    G_{i-1} \backslash N_{G_{i-1}}[i] & \cong P_{n-(2d+1)}^d \text{ for } i = 1, \ldots, d.\label{eq_3_2}
\end{align}
\begin{lem} \label{lem2} Assume that $n = k(d+2) + r$ where $0 \le r \le d$ and $k \ge 2$. Then 
$$\reg (R/I(C_n^d)) = k.$$    
\end{lem}
\begin{proof} By \cite[Corollary 3.4]{Ro} and \cite[Lemma 2.2]{K}, we have that $\reg (R/I(C_n^d)) \ge k$. We now prove by downward induction on $i \le d$ that $\reg (R/I(G_i)) \le k$. The base case where $i = d$ follows from Corollary \ref{cor_reg_powers_path} and Eq. \eqref{eq_3_1}. For the induction step, by Lemma \ref{lem_drop_v}, we have 
$$\reg (I(G_{i-1})) \le \max \{\reg ( I(G_i) ), \reg ( I(G_{i-1} \backslash N_{G_{i-1}} [i]) ) + 1\}.$$
The conclusion follows from induction, Corollary \ref{cor_reg_powers_path} and Eq. \eqref{eq_3_2}.
\end{proof}

The argument above does not apply to the case where $n = k(d+2) + d+1$ as we can see that $\reg (R/I(P_{n-(2d+1)}^d)) = k$ which is also the expected value for $\reg (R/I(C_n^d))$. We need to treat it separately in a series of lemmas. We further use the following notation.
\begin{align}
    I &= I(G) = I(G_0), J = I(G_1) \\
    K &= (x_{n-d+1}, \ldots, x_n, x_2, \ldots, x_{d+1}),\\
    A_j & = (x_\ell \mid \ell \in N_G(j) \backslash \{1,2\} ), \text{ for } j \neq 1,2\\
    A_2 & = (x_\ell \mid \ell \in N_G(2) \backslash \{1\}),\\
    B_j & = I(G \backslash N_G [ \{1,j\}] ), \text{ for } j \neq 1,
\end{align}

\begin{lem}\label{lem_s_1} With the notations above, we have $I = J + x_1 K$ and $J \cap K = L + x_2 M$, where 
\begin{align*}
     L & = \sum_{j=3}^{d+1} x_j(A_j + B_j) + \sum_{j=n-d+1}^n x_j(A_{j} +B_{j}), \\
    M &= A_2 + B_2.
\end{align*}
\end{lem}
\begin{proof} The equality $I = J + x_1K$ follows immediately from the definition. We now compute $J \cap K$. For $j \in N_G(1)$, we have 
    $$(x_j) \cap J = x_j (J :x_j) = x_j  \cdot \left ( ( x_\ell \mid \ell \in N_G(j) \backslash \{1\} ) + I(G \backslash N_G[j]) \right ).$$
Now for any $\ell \in N_G(1) \backslash N_G[j]$ and any $x_\ell x_k \in I(G \backslash N_G[j])$ we have $x_j x_\ell x_k \in x_\ell (x_m \mid m \in N_G(m) \backslash \{1\})$. The conclusion then follows from Lemma \ref{lem_colon_sum}.
\end{proof}
\begin{lem}\label{lem_s_2} Assume that $n = k(d+2) + d+1$ where $k \ge 2$. With the notations above, we have 
$$\reg (R/M) = k- 1.$$
\end{lem}
\begin{proof}
    Note that $G \backslash N_G[\{1,2\}] \cong P_{n-(2d+2)}^d$. By Lemma \ref{lem_s_1} and Corollary \ref{cor_reg_powers_path}, we deduce that 
    $$\reg (R/M) = \reg (R/I(P_{(k-1)(d+2) + 1}^d)) = \left \lfloor \frac{(k-1) (d+2) + d+1}{d+2} \right \rfloor  = k-1.$$
    The conclusion follows.
\end{proof}

\begin{lem}\label{lem_s_3} With the notations above, we have 
$$\reg (J) \le k+1.$$    
\end{lem}
\begin{proof} Let $H_1 = G_1 = G\backslash \{1\}$. We then define the graphs $H_j$ as follows. 
    \begin{align*}
        H_{2j} & = H_{2j-1} \backslash \{ n-d + j-1\} \text{ for } j = 1, \ldots, d-1,\\
        H_{2j+1} & = H_{2j} \backslash \{j+1\} \text{ for } j = 1, \ldots, d-1.             
    \end{align*}
    We prove by downward induction on $1 \le \ell \le 2d-1$ that $\reg (I(H_\ell)) \le k+1$. The base case where $k = 2d-1$ follows from Corollary \ref{cor_reg_powers_path} and the fact that  
    $$H_{2d-1} = G \backslash \{ n-d, \ldots, n-1, 1, \ldots, d\} \cong P_{n- (2d+1)}^d \cup \{n\},$$
    where $n$ is an isolated point. For the induction step, by applying Lemma \ref{lem_drop_v}, we need to prove that $\reg (I ( (H_{2j} \backslash N_{H_{2j}}[j+1]) ) ) \le k$ and $\reg ( I ( H_{2j-1} \backslash N_{H_{2j-1}} [n-d+j-1] ) ) \le k.$ 
    By definition, we have
    \begin{align*}
        H_{2j} \backslash N_{H_{2j}}[j+1] & = G \backslash (\{ n-d, \ldots, n-d+j-1\} \cup N_G[j+1]),\\
        H_{2j-1} \backslash N_{H_{2j-1}} [n-d+j-1] & = G \backslash ( \{1, \ldots, j\} \cup N_G[n-d+j-1]). 
    \end{align*}
    For simplicity, we denote by $F_{2j} = H_{2j} \backslash N_{H_{2j}} [j+1]$ and $F_{2j-1} = H_{2j-1} \backslash N_{H_{2j-1}}[n-d+j-1]$. We have that $F_{2j}$ is the induced subgraph of $G$ on $\{j+d+2, \ldots, n-d-1\} \cup \{n-d+j\}$. In particular, $F_{2j}$ is chordal. Hence, it suffices to prove that $\im(F_{2j}) \le k-1$. Let $\M$ be an induced matching of $F_{2j}$. If $n-d+j \notin \M$ then $\M$ is an induced matching of $P_{n - j - 2(d+1)}^d$ and the conclusion follows from Lemma \ref{lem_im_powers_paths}. If $n-d+j \in \M$ then dropping the edge containing $n-d+j$ from $\M$ we get an induced matching of $P_{n - j - 3(d+1)}$ and the conclusion follows from Lemma \ref{lem_im_powers_paths}. The conclusion that $\im(F_{2j-1}) \le k-1$ can be done similarly.    
\end{proof}
\begin{lem}\label{lem_s_4} Assume that $n = k(d+2) + d+1$ where $k \ge 2$. With the notations above, we have 
$$\reg (L) \le k + 1.$$    
\end{lem}
\begin{proof} We define the following ideals 
\begin{align*}
    L_1 & = L\\
    L_{2j} & = L_{2j-1} + (x_{n-d+j}) = L + (x_{n-d+1}, \ldots, x_{n-d+j}, x_3, \ldots, x_{j+1}),\\
    L_{2j+1} & = L_{2j} + (x_{j+2}) = L + (x_{n-d+1},\ldots,x_{n-d+j},x_3,\ldots,x_{j+2}). 
\end{align*}
We will prove by downward induction that $\reg (L_\ell) \le k+1$ for $1 \le \ell \le 2d-1$. For the base case where $\ell = 2d-1$, we have 
$$L_{2d-1} = x_n I (G \backslash N_G[\{1,n\}]).$$
Hence, $\reg (L_{2d-1}) = 1 + \reg (P_{n-2(d+1)}^d) = k+1$. For the induction step, by induction, it suffices to prove that $\reg (L_{2j-1} : x_{n-d+j}) \le k$ and $\reg (L_{2j} : x_{j+2}) \le k$. We have 
\begin{align*}
L_{2j-1} : x_{n-d+j} & = A_{n-d+j} + B_{n-d+j} + (x_{n-d+1},\ldots,x_{n-d+j-1},x_3,\ldots,x_{j+1})\\
L_{2j} : x_{j+2} & = A_{j+2} + B_{j+2} + (x_{n-d+1},\ldots,x_{n-d+j},x_3,\ldots,x_{j+1}).    
\end{align*}
In other words, they correspond to the edge ideals of some induced subgraphs of powers of paths. A similar argument to that of the proof of Lemma \ref{lem_s_3} applies to deduce our conclusion.
\end{proof}

\begin{lem}\label{lem_s_5} With the notations above, we have 
\begin{align*}
    L \cap M &= \sum_{j=n-d+2}^n x_j (A_j + B_j) + \sum_{j=3}^{d+1} x_j (A_j + B_j) \\
     & + x_{n-d+1} (( x_{n-d+2},\ldots,x_n)+  x_{d+2} (x_{n-2d+1}, \ldots,x_{n-d}, x_{d+3},\ldots,x_{2d+2}) + B_2).
\end{align*}    
\end{lem}
\begin{proof} By Lemma \ref{lem_s_1}, we have 
\begin{align*}
     L & = \sum_{j=3}^{d+1} x_j(A_j + B_j) + \sum_{j=n-d+1}^n x_j(A_{j} +B_{j}), \\
    M &= A_2 + B_2.
\end{align*}
Since $x_j(A_j + B_j)  \subseteq A_2$ for $j =3, \ldots, d+1$ and $j = n-d+2, \ldots, n$, we deduce that 
\begin{equation}\label{eq_in_1}
    x_j(A_j + B_j) \cap M = x_j(A_j + B_j)
\end{equation}
for $j = 3, \ldots, d+1$ and $j = n-d+2, \ldots, n$. We need to compute $x_{n-d+1} (A_{n-d+1} + B_{n-d+1}) \cap (A_2 + B_2)$. Note that $A_2 = (x_3, \ldots, x_{d+2}, x_{n-d+2}, \ldots, x_n)$ and $B_2 = G[d+3,\ldots,n-d]$. We denote by 
$$C = A_{n-d+1} + B_{n-d+1} = (x_{n-2d+1}, \ldots,x_{n-d},x_{n-d+2},\ldots,x_n) + I(G[d+2,\ldots,n-2d]).$$
For $j = 3,\ldots, d+1$, we have $x_{n-d+1} C \subseteq A_j + B_j$, hence 
\begin{equation}\label{eq_in_2}
    x_{n-d+1}C \cap (x_j)   \subseteq x_j (A_j + B_j).
\end{equation}
For $j = d+2$, we have 
\begin{equation}\label{eq_in_3}
    \begin{split}
    x_{n-d+1} C \cap (x_{d+2})  = x_{d+2} x_{n-d+1} ( & (x_{n-2d+1}, \ldots,x_{n-d},x_{n-d+2},\ldots,x_n,x_{d+3},\ldots,x_{2d+2}) \\
    &+  I ( G[2d+3,\ldots,n-2d]) ).    
    \end{split}
\end{equation}
For $j = n-d+2, \ldots, n$, we have 
\begin{equation}\label{eq_in_4}
x_{n-d+1} C \cap (x_j)  = (x_{n-d+1} x_j).    
\end{equation}
Now,
\begin{align*}
    B_2 \subseteq G[d+3,\ldots,n-2d] + (x_{n-2d+1},\ldots,x_{n-d}) \subseteq C. 
\end{align*}
Hence, 
\begin{equation}\label{eq_in_5}
    x_{n-d+1} C \cap B_2 = x_{n-d+1} B_2.
\end{equation} 
Since $I(G[2d+3,\ldots,n-2d]) \subseteq B_2$, by Lemma \ref{lem_colon_sum} and Eq. \eqref{eq_in_1}-\eqref{eq_in_5}, we deduce that 
\begin{align*}
    L \cap M & = \sum_{j=3}^{d+1} x_j (A_j + B_j) + \sum_{j=n-d+2}^n x_j (A_j + B_j) \\
    & + x_{n-d+1} ( x_{d+2} (x_{n-2d+1}, \ldots,x_{n-d}, x_{d+3},\ldots,x_{2d+2})) \\
    & + x_{n-d+1} ( x_{n-d+2},\ldots,x_n) + x_{n-d+1} B_2.
\end{align*}
The conclusion follows.
\end{proof}
\begin{lem}\label{lem_s_6} Assume that $n = k(d+2) + d+1$ where $k \ge 2$. With the notations above, we have $$\reg (L \cap M) \le k+ 1.$$    
\end{lem}
\begin{proof} For simplicity of notation, we denote by $P = L \cap M$. 
We note that $P + (x_{n-d+1}) = L + (x_{n-d+1})$. Hence, the arguments in the proof of Lemma \ref{lem_s_4} carry over except for the first step. It remains to prove that $\reg (P : x_{n-d+1}) \le k$. We have 
$$P : x_{n-d+1} = (x_{n-d+2},\ldots,x_n) + I(T)$$
where $T$ is the graph on $\{d+2, \ldots, n-d\}$ which is the union of the $d$-th power of the path from $d+3$ to $n-d$ and the edges $\{d+2, j\}$ with $j = n-2d+1, \ldots, n-d, d+3, \ldots, 2d+2$. It is easy to see that $T$ is chordal. It suffices to prove that $\im(T) \le k- 1$. The arguments as in the proof of Lemma \ref{lem_s_3} carry over to give the desired conclusion.
\end{proof}

\begin{lem}\label{lem_critical_case} Assume that $n = k(d+2) + d+1$ where $k \ge 2$. Then 
$$\reg (R/I(C_n^d)) = k.$$    
\end{lem}
\begin{proof} We use the notations as above. First, we prove that 
\begin{equation}\label{eq_3_10}
    \reg (J \cap K) \le k+1.
\end{equation} 
Indeed, by Lemma \ref{lem_intersection}, we have 
$$\reg (J \cap K) \le \max \{ \reg (L), \reg (M) + 1, \reg (L \cap M) \}.$$
The conclusion follows from Lemma \ref{lem_s_2}, Lemma \ref{lem_s_4}, and Lemma \ref{lem_s_6}.

Now, applying Lemma \ref{lem_intersection} we have 
$$ \reg (I) \le \max \{ \reg (J), \reg (K) + 1, \reg (J \cap K)\}.$$
Since $K$ is a linear ideal, $\reg (K) = 1$. The conclusion follows from Lemma \ref{lem_s_3} and Eq. \eqref{eq_3_10}.
\end{proof}

\begin{proof}[Proof of Theorem \ref{thm_reg_main}] For simplicity of notation, we denote by $G = C_n^d$. First, assume that $n \le 2d+2$. By \cite[Theorem 2.1]{Ro} and \cite[Corollary 3.4]{Ro}, $G$ is chordal and $\im(G) = 1$. By \cite[Theorem 1]{F}, $\reg (S/I(G)) = 1$. The case where $n = 2d+3$ follows from Lemma \ref{lem1}.

Now, assume that $n > 2d+3$. The conclusion follows from Lemma \ref{lem2} and Lemma \ref{lem_critical_case}.
\end{proof}


  \section*{Declarations}
  
  
  \subsection*{\bf Ethical approval}\ \\[-0.3cm]
  
  \noindent Not applicable.
  
  
  \subsection*{\bf Competing interests}\ \\[-0.3cm]
  
  \noindent The authors declare that there are no conflicts of interest.
  
  
  \subsection*{\bf Authors' contributions}\ \\[-0.3cm]
  
  \noindent All authors have contributed equally to this work.
  

  
  \subsection*{\bf Availability of data and materials}\ \\[-0.3cm]
  
  \noindent Data sharing is not applicable to this work, as no data sets were generated or analyzed during the current study.

\end{document}